\documentclass[oneside,english]{amsart}
\usepackage[T2A]{fontenc}
\usepackage[latin9]{inputenc}
\usepackage{listings}
\usepackage[a4paper]{geometry}
\usepackage{color}
\usepackage{babel}
\usepackage{amsthm}
\usepackage{amssymb}
\usepackage[unicode=true,pdfusetitle,
 bookmarks=true,bookmarksnumbered=false,bookmarksopen=true,bookmarksopenlevel=2,
 breaklinks=false,pdfborder={0 0 1},backref=false,colorlinks=true]
 {hyperref}

\makeatletter
\numberwithin{equation}{section}
\numberwithin{figure}{section}
\theoremstyle{plain}
\newtheorem{thm}{\protect\theoremname}
  \theoremstyle{remark}
  \newtheorem{claim}[thm]{\protect\claimname}
  \theoremstyle{definition}
  \newtheorem{defn}[thm]{\protect\definitionname}
  \theoremstyle{plain}
  \newtheorem{fact}[thm]{\protect\factname}
  \theoremstyle{definition}
  \newtheorem{example}[thm]{\protect\examplename}
  \theoremstyle{plain}
  \newtheorem{lem}[thm]{\protect\lemmaname}
  \theoremstyle{plain}
  \newtheorem{cor}[thm]{\protect\corollaryname}

\newcommand{\Sig}{\mathcal{S}}
\newcommand{\HM}{\mathrm{HM}}
\newcommand{\totaldeg}{\mathrm{deg}}
\newcommand{\poly}{\mathrm{poly}}
\newcommand{\sigidx}{\mathrm{index}}

\makeatother

  \providecommand{\claimname}{Claim}
  \providecommand{\corollaryname}{Corollary}
  \providecommand{\definitionname}{Definition}
  \providecommand{\examplename}{Example}
  \providecommand{\factname}{Fact}
  \providecommand{\lemmaname}{Lemma}
\providecommand{\theoremname}{Theorem}

\begin{document}

\title{Termination of original F5}
\begin{abstract}
The original F5 algorithm introduced by Faugère is formulated for
any homogeneous polynomial set input. The correctness of output is
shown for any input that terminates the algorithm, but the termination
itself is proved only for the case of input being regular polynomial
sequence. This article shows that algorithm correctly terminates for
any homogeneous input without any reference to regularity. The scheme
contains two steps: first it is shown that if the algorithm does not
terminate it eventually generates two polynomials where first is a
reductor for the second. But first step does not show that this reduction
is permitted by criteria introduced in F5. The second step shows that
if such pair exists then there exists another pair for which the reduction
is permitted by all criteria. Existence of such pair leads to contradiction.
\end{abstract}

\author{Vasily Galkin}

\address{Moscow State University}

\email{galkin-vv@ya.ru}

\maketitle

\section{Introduction}

The Faugère's F5 algorithm is known to be efficient method for Gröbner
basis computation but one of the main problems with it's practical
usage is lack of termination proof for all cases. The original paper
\cite{FaugereF5} and detailed investigations in \cite{F5-revisited}
states the termination for the case of reductions to zero absence,
which practically means termination proof for the case of input being
regular polynomial sequence. But for most input sequences the regularity
is not known, so this is not enough for practical implementations
termination proof. One of the approaches to solve this issue is adding
of additional checks and criteria for ensuring algorithm termination.
This approach is perfectly strict but the obtained result is termination
proof of a modified version of F5 algorithm which contain additional
checks and therefore can be more complex for implementation and possibly
slower for some input cases. Examples are \cite{Modifying-for-termination,Ars05applicationsdes,Gash:2008:ECG,ZobninGeneralization,Hashemi-ExtF5}.

The another approach is termination proof of custom F5-based algorithms
followed by attempt to reformulate original F5 in the terms of this
custom algorithm. The main problem of this approach arise during reformulation:
attempts to describe F5 in another terms may inadvertently introduce
some changes in behavior which are hard to discover but require additional
proofs to show equivalence. For example \cite{2012GrbTermination}
proofs the termination of the F5GEN algorithm which differs from original
F5 by absence of criteria check during reductor selection. The \cite{HuangConception}
gives the proof of TRB-F5 algorithm termination which has two main
differences realized by the author with great help of discussions
with John Perry. The first is another rule building scheme which eventually
lead to the ordering by signature the rules in $Rule$ array during
TRB-F5 execution. The second is the absence of applying in TRB-F5
the normal form $\varphi$ before reduction, which leads to an effect
opposite to the difference with F5GEN: criteria checks are applied
for elements with greater signature index which are used in original
F5 in the normal from operator as reductors without the checks. The
author thinks that these algorithms may be changed to exactly reproduce
the original F5 behavior and the termination proofs can be applied
to such changed versions. But the approach with algorithms equivalent
to F5 has a drawback: it makes harder to understand how the theorems
used in termination proof can be expressed in terms of original F5
behavior.

This paper introduces another approach for termination proof of original
algorithm without any modifications. The first step of proof is based
on the idea of S-pair-chains which are introduced in this paper. The
second step of proof is based on the method described in Theorem 21
of \cite{F5C} for the proof of F5C algorithm correctness: the representation
of an S-polynomial as the sum of multiplied polynomials from set computed
by F5C can be iteratively rewritten using replacements for S-pairs
and rejected S-pair parts until a representation with certain good
properties is achieved after finitely many steps.

This article shows that the hypothesis of this method can be weakened
to apply it for the set at any middle stage of F5 computations and
the conclusions can be strengthened to use them for termination proof.
The paper is designed as alternative termination proof for exact algorithm
described in \cite{FaugereF5}, so the reader is assumed to be familiar
with it and all terminology including names for algorithm steps are
borrowed from there.

\section{Possibilities for infinite cycles in F5}

\subsection{Inside \texttt{AlgorithmF5}: $d$ growth}
\begin{claim}
\label{iterations-d_grow}If the number of \textbf{while} cycle iterations
inside \texttt{AlgorithmF5} is infinite then the $d$ value infinitely
grow.\end{claim}
\begin{proof}
Let's suppose that there is an input $\left\{ f_{1},\ldots,f_{m}\right\} $
over $\mathcal{K}[x_{1},\ldots,x_{n}]$ for which original F5 does
not terminate, and that it is shortest input of such kind -- the algorithm
terminates on shorter input $\left\{ f_{2},\ldots,f_{m}\right\} $.
This means that last iteration of outer cycle in \texttt{incrementalF5}
does not terminate, so last call to \texttt{AlgorithmF5} does not
terminate. To investigate this we need to study how the total degree
$d$ can change during execution of the cycle inside \texttt{AlgorithmF5}.
Let's call $d_{j}$ the value of $d$ on $j$-th cycle iteration
and extend it to $d_{0}=-1$. The simple property of $d_{j}$ is it's
non-strict growth: $d_{j}\geqslant d_{j-1}$. It holds because on
the $j-1$-th iteration all polynomials in $R_{d}$ have degree $d_{j-1}$
and therefore all new generated critical pairs have degree at least
$d_{j-1}$. Now suppose that $j$ is number of some fixed iteration.
At the iteration $j$ all critical pairs with degree $d_{j}$ are
extracted from $P$. After call to \texttt{Reduction} some new critical
pairs are added to $P$ in the cycle over $R_{d}$. There exist a
possibility that some of them has degree $d_{j}$. We're going to
show that all such critical pairs do not generate S-polynomials in
the next iteration of algorithm because they are discarded.

For each new critical pair $[t,u_{1},r_{1},u_{2,}r_{2}]$ generated
during iteration $j$ at least one of the generating polynomials belong
to $R_{d}$ and no more than one belong to $G_{i}$ at the beginning
of the iteration. All polynomials in $R_{d}$ are generated by \texttt{Reduction}
function by appending single polynomials to $Done$. So we can select
from one or two $R_{d}$-belonging generators of critical pair a polynomial
$r_{k}$ that was added to $Done$ later. Then we can state that the
other S-pair part $r_{3-k}$ was already present in $G\cup Done$
at the moment of $r_{k}$ was added to $Done$. So the \texttt{TopReduction}
tries to reduce $r_{k}$ by $r_{3-k}$ but failed to do this because
one of \texttt{IsReducible} checks (a) - (d) forbids this.

From the other hand for critical pairs with degree equal to $d_{j}$
we have $u_{k}=1$ because total degree of critical pair is equal
to total degree of it's generator $r_{k}$. This means that value
$u_{3-k}$ is equal to $\frac{\HM(r_{k})}{\HM(r_{3-k})}$ so the \texttt{IsReducible}'s
rule (a) allow reduction $r_{k}$ by $r_{3-k}$. It follows that only
checks (b) - (d) are left as possibilities.

Suppose that reduction was forbidden by (b). This means that there
is a polynomial in $G_{i+1}$ that reduces $u_{3-k}\Sig(r_{3-k})$.
For our case it means that in the \texttt{CritPair} function the same
check $\varphi(u_{3-k}\Sig(r_{3-k}))=u_{3-k}\Sig(r_{3-k})$ fails
and such critical pair would not be created at all. So the rule (b)
can't forbid reduction too.

Suppose that reduction was forbidden by (c). This means that there
is a rewriting for the multiplied reductor. So for our case it means
that \texttt{Rewritten?}$\left(u_{3-k},r_{3-k}\right)$ returns true
at the moment of \texttt{TopReduction} execution, so it still returns
true for all algorithm execution after this moment because rewritings
do not disappear.

Suppose that reduction was forbidden by (d). The pseudo code in \cite{FaugereF5}
is a bit unclear at this point, but the source code of procedure \texttt{FindReductor}
attached to \cite{F5-revisited} is more clear and states that the
reductor is discarded if both monomial of signature and index of signature
are equal to those of polynomial we're reducing (it's signature
monomial is \texttt{r{[}k0{]}{[}1{]}} and index is \texttt{r{[}k0{]}{[}2{]}}
in the code):

\begin{lstlisting}[basicstyle={\ttfamily},tabsize=4]
if (ut eq r[k0][1]) and (r[j][2] eq r[k0][2]) then
	// discard reductor by criterion (d)
	continue;
end if;
\end{lstlisting}
For our case it means that signatures of $r_{k}$ and $u_{3-k}r_{3-k}$
are equal. This leads to fact that \texttt{Rewritten?}$\left(u_{3-k},r_{3-k}\right)$
returns true after adding rule corresponding to $r_{k}$ because $u_{3-k}\cdot r_{3-k}$
is rewritable by $1\cdot r_{k}$. So like in case (c) \texttt{Rewritten?}$\left(u_{3-k},r_{3-k}\right)$
returns true at the moment of \texttt{TopReduction} execution.

Now consider \texttt{Spol} function execution for some S-pair with
total degree $d_{j}$ generated during iteration $j$. It executes
in $j+1$ iteration of \texttt{AlgorithmF5} cycle which is far after
\texttt{TopReduction} execution for $r_{k}$ in algorithm flow so
for both cases (c) and (d) call to \texttt{Rewritten?}$\left(u_{3-k},r_{3-k}\right)$
inside \texttt{Spol} returns true. It means that at the $j+1$ step
no S-pair with total degree $d_{j}$ can add polynomial to $F$.

In the conclusion we have: 
\begin{itemize}
\item the first possibility of $d_{j+1}$ and $d_{j}$ comparison is $d_{j+1}=d_{j}$.
In this case $F$ is empty on $j+1$ iteration and therefore $P$
does not contain any pairs with degree $d_{j}$ after $j+1$ iteration's
finish. So $d_{j+2}>d_{j+1}$. 
\item the other possibility is $d_{j+1}>d_{j}$.
\end{itemize}
In conjunction with non-strict growth this gives $\forall j\,\, d_{j+2}>d_{j}$
which proves the claim \ref{iterations-d_grow}.
\end{proof}

\subsection{Inside \texttt{Reduction}: $ToDo$ finiteness}
\begin{claim}
\label{Every_cycle_iteration_finish}Every cycle iteration inside
\texttt{AlgorithmF5} does terminate, in particular all calls to \texttt{Reduction}
terminate.\end{claim}
\begin{proof}
The calls to \texttt{AlgorithmF5} corresponding to polynomials $f_{2},\dots,f_{m}$
are known to terminate, so we will study the only left call of \texttt{AlgorithmF5}
corresponding to processing of input sequence item $f_{1}$. Firstly
we need to get some facts about polynomials in $ToDo$ and $Rule$
sets inside $j$-th iteration of that call to \texttt{AlgorithmF5}.
All the critical pairs created initially by \texttt{CritPair} inside
the \texttt{AlgorithmF5} have greater S-pair part with signature index
$1$. All other critical pairs are generated with signature index
corresponding to $ToDo$ elements moved to $Done$ set. All elements
of $ToDo$ are generated either from critical pairs or in the \texttt{TopReduction}
procedure. The polynomials generated by \texttt{TopReduction} has
signatures greater than the signature of polynomial the function tries
to top-reduce which is $ToDo$ element. So, there is no place in algorithm
flow where a polynomial or critical pair with signature index different
from 1 can be generated in the \texttt{AlgorithmF5} call. From the
other hand all polynomials inside $ToDo$ has the same total degree
$d_{j}$. Together with index equality this shows that the total degree
of signature monomials is equal to the $d_{j}-\totaldeg(f_{1})$ for
all $ToDo$ elements.

Every addition of polynomial in the $Rule$ set correspond to the
addition of $ToDo$ element. So, the elements added to $Rule$ in
$j$-th iteration have total degree equal to $d_{j}$. Combining this
with non-strict $d_{j}$ growing we get that at $j$-th iteration
all elements of $Rule$ with signature index 1 have total degree $\leqslant d_{j}$
and total degree of signature $\leqslant d_{j}-\totaldeg(f_{1})$.
In addition this gives the fact about order of $Rule$ elements with
signature index 1: their total degrees are non-strictly increasing.
\begin{defn}
The reduction of labeled polynomial $r_{k}$ with a labeled polynomial
$r_{m}$ is called \emph{signature-safe} if $\Sig(r_{k})\succ t\cdot\Sig(r_{m})$,
where $t=\frac{\HM(r_{k})}{\HM(r_{m})}$ is monomial multiplier of
reductor. The reductor corresponding to signature-safe reduction is
called \emph{signature-safe reductor}. 
\end{defn}
The algorithm performs only signature-safe reductions: the \texttt{TopReduction}
function performs reduction by non-rejected reductor if it is signature-safe
and adds new element to $ToDo$ otherwise. The elements of $ToDo$
are processed in signature increasing order, so no elements of $G\cup Done$
has signature greater than signature of polynomial $r_{k}$ being
reduced in \texttt{TopReduction}. If the reductor $r_{m}$ has $\totaldeg(r_{m})=\totaldeg(r_{k})$
we have $t=1$ and $\Sig(r_{k})\succ\Sig(r_{m})$ which ensures the
signature-safety of such reduction because the case $\Sig(r_{k})=\Sig(r_{m})$
would have been rejected by (d) check in \texttt{IsReducible}. So
$\Sig(r_{k})\prec t\cdot\Sig(r_{m})$ is possible only when $\totaldeg(r_{m})<\totaldeg(r_{k})$
and all additions in $ToDo$ in \texttt{TopReduction} correspond to
such situation. The signature of polynomial added such way is $t\cdot\Sig(r_{m})$
and the fact that $r_{m}$ was not rejected by \texttt{Rewritten?}
check in \texttt{IsReducible} ensures that no polynomial with signature
$t\Sig(r_{m})$ were generated yet because such polynomial would have
rule corresponding to it in $Rule$ with greater total degree than
rule corresponding to $r_{m}$ and $r_{m}$ would be rejected in \texttt{IsReducible}. 

We want to show that only possible algorithm non-termination situation
correspond to the case of infinite $d_{j}$ growth. We showed that
non-termination leads to \texttt{AlgorithmF5} does not return, and
that it can't stuck in iterations with same $d$ value. So, the only
possibilities left are infinite $d$ growth and sticking inside some
iteration. We are going to show that such sticking is not possible.
The \texttt{AlgorithmF5} contains 3 cycles:
\begin{itemize}
\item \textbf{for} cycle inside \texttt{Spol} does terminate because it's
number of iterations is limited by a count of critical pairs which
is fixed at cycle beginning
\item \textbf{for} cycle inside \texttt{AlgorithmF5} iterating over $R_{d}$
elements also does terminate because count of $R_{d}$ elements is
fixed at cycle beginning
\item the most complex case is the \textbf{while} cycle inside \texttt{Reduction}
which iterates until $ToDo$ becomes empty. The $ToDo$ set is initially
generated by \texttt{Spol} and then extended by new elements during
\texttt{TopReduction} execution. \texttt{Spol} generates finite number
of elements because it terminates and the \texttt{TopReduction} adds
elements with distinct signatures having index 1 so their number is
limited by the count of different signatures of total degree $d_{j}-\totaldeg(f_{1})$,
so only finite number of elements is added in $ToDo$. We will show
that all types of steps that perform \texttt{Reduction} can be performed
only finitely number of times:

\begin{itemize}
\item the step when \texttt{IsReducible} returns empty set correspond to
transferring $ToDo$ element in $Done$ and the number of such steps
is limited by number of elements added in $ToDo$
\item the step when \texttt{IsReducible} returns reductor which is not signature-safe
correspond to adding new element in $ToDo$ and the number of such
steps is limited by the number of possible additions.
\item the step when \texttt{IsReducible} returns reductor which is signature-safe
correspond to reduction of some $ToDo$ element. This can be done
only finite number of times because there are finitely many polynomials
added in $ToDo$ and no polynomial can be top-reduced infinite number
of times because its $\HM$ is $\prec$-decreasing during reduction
and monomials are well-ordered with $\prec$.
\end{itemize}
\end{itemize}
We got that all of the cycles inside \texttt{AlgorithmF5} do finish
and the claim \ref{Every_cycle_iteration_finish} is proved.
\end{proof}
This gives the result about algorithm behavior for the non-terminated
case:
\begin{claim}
\label{d-does-grow}If the algorithm does not terminate for some input
then the value of $d$ infinitely grow during iterations.\end{claim}
\begin{proof}
Follows from combination of claims \ref{iterations-d_grow} and \ref{Every_cycle_iteration_finish}.
\end{proof}

\section{S-pair-chains}

The claim \ref{d-does-grow} shows that algorithm non termination
leads to existence of infinite sequence of nonzero labeled polynomials
being added to $G_{i}$ and the total degrees of polynomials in the
sequence infinitely grow. So, in this case algorithm generates an
infinite sequence of labeled polynomials $\left\{ r_{1},r_{2},\ldots,r_{m},\ldots,r_{l},\ldots\right\} $
where $r_{1},\ldots,r_{m}$ correspond to $m$ input polynomials and
other elements are generated either in \texttt{Spol} or in \texttt{TopReduction}.
In both cases new element $r_{l}$ is formed as S-polynomial of two
already existing polynomials already present in the list. We will
write $l^{*}$ and $l_{*}$ for the indexes of the polynomials used
to generate $l$-th element and $\overline{u_{l}}$, $\underline{u_{l}}$
for monomials they are multiplied. Note that $l^{*}$ correspond to
the part with greater signature: $\poly(r_{l})=\overline{u_{l}}\poly(r_{l^{*}})-\underline{u_{l}}\poly(r_{l_{*}})$
and $\Sig(r_{l})=\overline{u_{l}}\Sig(r_{l^{*}})\succ\underline{u_{l}}\Sig(r_{l_{*}})$.
The $\poly(r_{l})$ value can further change inside \texttt{TopReduction}
to the polynomial with a smaller $\HM$, but the $\Sig(r_{l})$ does
never change after creation. Now, we want to select an infinite sub-sequence
$\left\{ r_{k_{1}},r_{k_{2}},\ldots,r_{k_{n}},\ldots\right\} $ in
that sequence with the property that $r_{k_{n}}$ is an S-polynomial
generated by $r_{k_{n-1}}=r_{k_{n}^{*}}$ and some other polynomial
corresponding to smaller by signature S-pair part, so $\Sig(r_{k_{n}})=\overline{u_{k_{n}}}\Sig(r_{k_{n-1}})$
and 
\begin{equation}
\Sig(r_{k_{n-1}})|\Sig(r_{k_{n}}).\label{eq:s-pair-chain-def}
\end{equation}

\begin{defn}
Finite or infinite labeled polynomial sequence which successive elements
satisfy property \ref{eq:s-pair-chain-def} will be called \emph{S-pair-chain}.
\end{defn}
Every generated labeled polynomial $r_{l}$ has an finite S-pair-chain
ending with that polynomial. This chain can be constructed in reverse
direction going from it's last element $r_{l}$ by selecting every
step from a given polynomial $r_{n}$ a polynomial $r_{n^{*}}$ which
was used to generate $r_{n}$ as S-polynomial. The resulting S-pair-chain
has the form $\{r_{q},\ldots,r_{l^{**}},r_{l^{*}},r_{l}\}$ where
all polynomials has the same signature index $q=\sigidx(r_{l})$ and
the first element is the input polynomial of that index.

The first fact about S-pair-chains is based on the rewritten criteria
and consists in the following theorem.
\begin{thm}
Every labeled S-polynomial can participate as the first element only
in finite number of S-pair-chains of length 2.\end{thm}
\begin{proof}
The \texttt{AlgorithmF5} computes S-polynomials in 2 places: in procedure
\texttt{SPol} and in the procedure \texttt{TopReduction}. It's important
that in both places the \texttt{Rewritten?} check for the part of
S-polynomial with greater signature is performed just before the S-polynomial
is constructed. In the first case the \texttt{SPol} is checking that
itself, in the \texttt{TopReduction} the check is in the \texttt{IsReducible}
procedure. And in both cases the computed S-polynomial is immediately
added to the $Rule$ list as the newest element. So, at the moment
of the construction of S-polynomial with signature $s$ we can assert
that the higher part of S-pair correspond to the newest rule with
signature dividing $s$ -- this part even may be determined by $Rule$
list and $s$ without knowing anything other about computation.

Consider arbitrary labeled polynomial $r_{L}$ with signature $\Sig(r_{L})=s$
and an ordered by generating time subset $\{r_{l_{1}},\ldots,r_{l_{i}},\ldots\}$
of labeled polynomials with signatures satisfying $\Sig(r_{l_{i}})=v_{i}\Sig(r_{L})$.
From the signature divisibility point of view all of the possibly
infinite number of pairs $\{r_{L},r_{l_{i}}\}$ can be S-pair-chains
of length 2. But the ideal $\left(v_{i}\right)$ in $T$ is finitely
generated by Dickson's lemma, so after some step $i_{0}$ we have
$\forall i>i_{0}\,\exists j\leqslant i_{0}$ such that $v_{j}|v_{i}$.
So $\forall i>i_{0}$ the sequence $\{r_{L},r_{l_{i}}\}$ is not S-pair-chain
because $\Sig(r_{L})\cdot v_{i}$ is rewritten by $\Sig(r_{l_{j}})\cdot\frac{v_{i}}{v_{j}}$
and no more than $i_{0}$ S-pair-chains of length 2 with first element
$r_{L}$ exist.\end{proof}
\begin{defn}
The finite set of ends of 2-length S-pair-chains starting with $r_{L}$
will be called \emph{S-pair-descendants} of $r_{L}$.\end{defn}
\begin{thm}
If the algorithm does not terminate for some input then there exists
infinite S-pair-chain $\{h_{i}\}$.\end{thm}
\begin{proof}
Some caution is required while dealing with infinities, so we give
the following definition. 
\begin{defn}
The labeled polynomial $r_{l}$ is called \emph{chain generator} if
there exist infinite number of different finite S-pair-chains starting
with $r_{l}$.
\end{defn}
If the algorithm does not terminate the input labeled polynomial $r_{1}=(f_{1},1F_{1})$
is chain generator because every labeled polynomial $r_{l}$ generated
in the last non-terminating call to \texttt{AlgorithmF5} has signature
index 1 so there is an S-pair-chain $\{r_{1},\ldots,r_{l^{**}},r_{l^{*}},r_{l}\}$.

Now assume that some labeled polynomial $r_{l}$ is known to be a
chain generator. Then one of the finite number of S-pair-descendants
of $r_{l}$ need to be a chain generator too, because in the other
case the number of different chains of length greater than 2 coming
from $r_{l}$ was limited by a finite sum of the finite counts of
chains coming from every S-pair-descendant, and the finite number
of length 2 chains coming from $r_{l}$. So, if labeled polynomial
$r_{l}$ is chain-generator, we can select another chain generator
from it's S-pair-descendants. In a such way we can find infinite
S-pair-chain starting with $r_{1}$ and consisting of chain generators
which proves the theorem.
\end{proof}
For the next theorem we need to introduce monomial quotients order
by transitively extending the monomial ordering: $\frac{m_{1}}{m_{2}}>_{q}\frac{m_{3}}{m_{4}}\Leftrightarrow m_{1}m_{4}>m_{3}m_{2}$.
\begin{thm}
\label{thm:f_g_3_props}If the algorithm does not terminate for some
input then after some finite step the set $G$ contains a pair of
labeled polynomials $f',f$ with $f$ generated after $f'$ that satisfies
the following 3 properties:

\[
\HM(f')|\HM(f),
\]

\[
\frac{\HM(f')}{\Sig(f')}>_{q}\frac{\HM(f)}{\Sig(f)},
\]

\[
\Sig(f')|\Sig(f).
\]
\end{thm}
\begin{proof}
For working with S-pair-chains it is important that the polynomial
can never reduce after it was used for S-pair generation as higher
S-pair part. That's true because all polynomials that potentially
can reduce are stored in set $ToDo$, but all polynomials that are
used as higher S-pair part are stored in $G$ or in $Done$. So we
may state that the polynomial $h_{n}$ preceding polynomial $h_{n+1}$
in the S-pair-chain keeps the same $\poly(h_{n})$ value after it
was used for some S-pair generation and we can state that
\[
\poly(h_{n+1})=c\frac{\Sig(h_{n+1})}{\Sig(h_{n})}\poly(h_{n})+g_{n},
\]
where $g_{n}$ is the polynomial corresponding to smaller part of
S-pair used to generate $h_{n+1}$ from $h_{n}$ and satisfy:
\begin{equation}
\HM(h_{n+1})<\HM\left(\frac{\Sig(h_{n+1})}{\Sig(h_{n})}h_{n}\right)=\HM(g_{n}),\,\Sig(h_{n+1})=\Sig\left(\frac{\Sig(h_{n+1})}{\Sig(h_{n})}h_{n}\right)\succ\Sig(g_{n}).\label{eq:spair-chain}
\end{equation}
From the first inequality in \ref{eq:spair-chain} we can get $\frac{\HM(h_{n})}{\Sig(h_{n})}>_{q}\frac{\HM(h_{n+1})}{\Sig(h_{n+1})}$,
so in the S-pair-chain the quotients $\frac{\HM(h_{i})}{\Sig(h_{i})}$
are strictly descending according to quotients ordering. This fact
can't be used directly to show chains finiteness because unlike the
ordering of monomials the ordering of monomial quotients is not well
ordering -- for example the sequence $\frac{x}{x}>_{q}\frac{x}{x^{2}}>_{q}\cdots>_{q}\frac{x}{x^{n}}>_{q}\cdots$
is infinitely decreasing.

There is two possible cases for relation between $\HM$'s of consecutive
elements. We have $\Sig(h_{n})|\Sig(h_{n+1})$, so they either have
equal signatures or $\totaldeg(h_{n})<\totaldeg(h_{n+1})$. For the
first case $\HM(h_{n+1})<\HM(h_{n})$ with the equal total degrees
for the other case $\HM(h_{n+1})>\HM(h_{n})$ because total degrees
of $\HM$'s are different. So the sequence of infinite S-pair-chain
$\HM$'s consists of blocks with fixed total degrees where $\HM$'s
inside a block are strictly decreasing. Block lengths can be equal
to one and the total degree of blocks are increasing. This leads to
the following properties: S-pair-chain $\{h_{i}\}$ can't contain
elements with equal $\HM$'s and $\HM(h_{i})|\HM(h_{j})$ is possible
only for $i<j$ and $\totaldeg(h_{i})<\totaldeg(h_{j})$.

This allows us to use technique analogous to Proposition 14 from \cite{TheF5Revised}:
consider $\HM$'s of infinite S-pair-chain $\{h_{i}\}$. They form
an infinite sequence in $T$, so by Dickson's lemma there exists
2 polynomials in sequence with $\HM(h_{i})|\HM(h_{j})$. Therefore
from the previous paragraph we have $i<j$ and with the S-pair-chain
properties we get $\Sig(h_{i})|\Sig(h_{i+1})|\cdots|\Sig(h_{j})$
and $\frac{\HM(h_{i})}{\Sig(h_{i})}>_{q}\frac{\HM(h_{i+1})}{\Sig(h_{i+1})}>_{q}\cdots>_{q}\frac{\HM(h_{j})}{\Sig(h_{j})}$,
so we take $f'=h_{i}$ and $f=h_{j}$.
\end{proof}
The last property about signature division from the theorem claim
is the consequence of dealing with S-pair-chains and is not used in
the following. But the first two properties are used to construct
a signature-safe reductor.
\begin{fact}
If no polynomials are rejected by criteria checks (b) and (c) inside
\texttt{IsReducible} the algorithm does terminate.\end{fact}
\begin{proof}
The above proof of theorem \ref{thm:f_g_3_props} does not rely on
any correspondence between orderings on signatures and terms. But
the original F5 algorithm uses the same ordering for both cases and
now we utilize this fact and make a transition from one to other to
get relation on signatures for polynomials from theorem \ref{thm:f_g_3_props}
claim: 
\[
\Sig(g)\succ t\cdot\Sig(f),\mbox{ where }t=\frac{\HM(g)}{\HM(f)}\in T.
\]
The last inequality with $\HM$'s division property from theorems
result shows that $tf$ can be used as a reductor for $g$ in \texttt{TopReduction}
from the signature point of view -- i.e. it satisfy checks (a) and
(d) inside \texttt{IsReducible} and it's signature is smaller. In
the absence of criteria checks (b) and (c) this would directly lead
to contradiction because at the time $g$ was added to the set $G$
labeled polynomial $f$ already had been there, so the \texttt{TopReduction}
should had reduced $g$ by $f$.
\end{proof}
But the existence of criteria allow the situation in which $tf$ is
rejected by criteria checks in (b) or (c) inside \texttt{IsReducible}.
The idea is to show that even in this case there can be found another
possible reductor for $g$ that is not rejected and anyway lead to
contradiction and the following parts of paper aim to prove it.

\section{S-pairs with signatures smaller than $\Sig(g)$}

In this and following sections $g$ is treated as some fixed labeled
polynomial with signature index 1 added to $Done$ in some algorithm
iteration. Let us work with algorithm state just before adding $g$
to $Done$ during call to \texttt{AlgorithmF5} with $i=1$. Consider
a finite set $G_{1}\cup Done$ of labeled polynomials at that moment.
This set contains positions of labeled polynomials in $R$, so it's
elements can be ordered according to position in $R$ and written
as an ordered integer sequence $G_{g}=\{b_{1},\ldots,b_{N}\}$ with
$b_{j}<b_{j+1}$. It should be noted that this order does correspond
to the order of labeled polynomials in the sequence produced by concatenated
rule arrays $Rule[m]:Rule[m-1]:\cdots:Rule[1]$ because addition of
new polynomial to $R$ is always followed by addition of corresponding
rule. But this order may differ from the order polynomials were added
to $G_{1}\cup Done$ because polynomials with same total degree are
added to $Done$ in the increasing signature order, while the addition
polynomials with same total degree to $R$ is performed in quite random
order inside \texttt{Spol} and \texttt{TopReduction} procedures. For
the simplicity we will be speaking about labeled polynomials $b_{j}$
from $G_{g}$, assuming that $G_{g}$ is not the ordered positions
list but the ordered list of labeled polynomials themselves corresponding
to those positions. In this terminology we can say that all input
polynomials $\left\{ f_{1},\ldots,f_{m}\right\} $ do present in $G_{g}$,
because they all present in $G_{1}$ at the moment of its creation.

S-pairs can be processed in a different ways inside the algorithm
but the main fact we need to know about their processing is encapsulated
in the following properties which correspond to the properties used
in Theorem 21 in \cite{F5C} but are taken during arbitrary algorithm
iteration without requirements of termination.
\begin{thm}
\label{thm:Exist-gg-repr}At the moment of adding $g$ to $Done$
every S-pair of $G_{g}$ elements which signature is smaller than
$\Sig(g)$ satisfies one of three properties:\end{thm}
\begin{enumerate}
\item S-pair has a part that is rejected by the normal form check $\varphi$
(in \texttt{CritPair} or in \texttt{IsReducible}). Such S-pairs will
be referenced as \emph{S-pairs with a part that satisfies F5 criterion}. 
\item S-pair has a part that is rejected by the \texttt{Rewritten?} check
(in \texttt{SPol} or in \texttt{IsReducible}). Such S-pairs will be
referenced as \emph{S-pairs with a part that satisfies Rewritten criterion}.
\item S-pair was not rejected, so it's S-polynomial was signature-safe
reduced by $G_{g}$ elements and the result is stored in $G_{g}$.
Such S-pairs will be referenced as \emph{S-pairs with a computed $G_{g}$-representation}.\end{enumerate}
\begin{proof}
S-pairs of $G_{g}$ elements are processed in two paths in the algorithm.
The main path is for S-pairs with total degree greater than total
degrees of polynomials generated it. Such S-pairs are processed in
the following order:
\begin{itemize}
\item in the \texttt{AlgorithmF5} they are passed in \texttt{CritPair} function
while moving elements to $G_{i}$ from $R_{d}=Done$ or while processing
input polynomial $r_{i}$. 
\item The \texttt{CritPair} function either discards them by normal form
check $\varphi$ or adds to $P$
\item The S-pair is taken from $P$ and passed to \texttt{SPol} function
\item The \texttt{SPol} function either discards them by \texttt{Rewritten?}
check or adds S-polynomial to $F=ToDo$
\item At some iteration the \texttt{Reduction} procedure takes S-polynomial
from $ToDo$, performs some signature-safe reductions and adds result
to $Done$.
\end{itemize}
The other processing path is for special S-pairs corresponding to
reductions forbidden by algorithm -- the case when S-pair is generated
by polynomials $r_{l^{*}}$ and $r_{l_{*}}$ such that $\HM(r_{l^{*}})|\HM(r_{l_{*}})$
so that S-polynomial has a form $\overline{u_{l}}\cdot\poly(r_{l^{*}})-1\cdot\poly(r_{l_{*}})$.
Such situation is possible for two $G_{g}$ elements if reduction
of $r_{l_{*}}$ by $r_{l^{*}}$ was forbidden by signature comparison
in \texttt{TopReduction} or by checks in \texttt{IsReducible}. So
for this case the path of S-pair ``processing'' is the following:
\begin{itemize}
\item The S-pair part $\overline{u_{l}}\cdot r_{l^{*}}$ is checked in \texttt{IsReducible}.
(a) is satisfied because $\HM(r_{l^{*}})|\HM(r_{l_{*}})$. It can
be rejected by one of the other checks:

\begin{itemize}
\item Rejection by check (b) correspond to the normal form check $\varphi$
for $\overline{u_{l}}\cdot r_{l^{*}}$
\item Rejection by check (c) correspond to the \texttt{Rewritten?} check
for $\overline{u_{l}}\cdot r_{l^{*}}$
\item Rejection by check (d) means that either $\overline{u_{l}}\cdot r_{l^{*}}$
or $1\cdot r_{l_{*}}$ can be rewritten by other, so if S-pair was
not rejected by check (c) this type of rejection means that S-pair
part $1\cdot r_{i_{1}}$ fails to pass \texttt{Rewritten?} check.
\end{itemize}
\item The non-rejected in \texttt{IsReducible} S-pair is returned in the
\texttt{TopReduction}. Signature comparison in \texttt{TopReduction}
forbids the reduction of $r_{l_{*}}$ by $r_{l^{*}}$ and returns
computed S-polynomial corresponding to the S-pair in the set $ToDo_{1}$
\item The \texttt{Reduction} procedure add this polynomial in $ToDo$
\item The last step is equal to for both processing paths: at some iteration
the \texttt{Reduction} procedure takes S-polynomial from $ToDo$,
performs some signature-safe reductions and adds result to $Done$
\end{itemize}
It can be seen that after S-pair processing termination every S-pair
is either reduced and added to $Done$ or one of its S-pair parts
is rejected by normal form $\varphi$ or \texttt{Rewritten?} check.
Some S-pairs can be processed by processing paths multiple times,
for example this is done in the second iteration inside \texttt{AlgorithmF5}
with same $d$ value. If the pair was rejected during first processing
it will be rejected the same way during next attempt. If the first
time processing adds polynomial to $Done$ the pair will be rejected
in next attempts by \texttt{the Rewritten?} check with that polynomial.
So all but the first processing attempts are insignificant.

The processing path is not a single procedure and for the case of
algorithm infinite cycling some S-pairs are always staying in the
middle of the path having S-pair queued in $P$ or S-polynomial in
$ToDo$. So we have to select S-pairs which processing is already
finished at the fixed moment we studying. The elements from $P$ and
$ToDo$ in \texttt{AlgorithmF5} and \texttt{Reduction} procedures
are taken in the order corresponding to growth of their signatures.
So S-pairs with signature smaller than $\Sig(g)$ can be split in
the following classes:
\begin{itemize}
\item S-pairs with signature $w$ such that $\sigidx(w)>\sigidx(\Sig(g))=1$.
They were processed on previous calls of \texttt{AlgorithmF5}.
\item S-pairs with signature $w$ such that $\sigidx(w)=\sigidx(\Sig(g))=1,\totaldeg(w)<\totaldeg(\Sig(g))$.
They were processed on previous iterations inside the call to \texttt{AlgorithmF5}
that is processing $g$.
\item S-pairs with signature $w$ such that $\sigidx(w)=\sigidx(\Sig(g))=1,\totaldeg(w)=\totaldeg(\Sig(g)),w\prec\Sig(g)$.
They were processed on previous iterations inside the call to Reduce
that is processing $g$.
\end{itemize}
S-pairs from these classes can't be at the middle of processing path
because at the studied state of algorithm the processing is just finished
for $g$ so $P$ and $ToDo$ sets does not contain any non-processed
elements with signatures smaller $\Sig(g)$. The only left thing to
show is proof that processing was started at least one time for all
S-pairs from theorem claim. This is true for first two classes: the
processing of corresponding S-pairs was started at least one time
with the call to \texttt{CritPair} inside \texttt{AlgorithmF5} just
before the greatest of S-pair generators was added to $G$. For S-pairs
of third class the situation depend on the total degrees of its generators.
If both generators of S-pair have total degrees $<\totaldeg(g)$ then
its processing is started in \texttt{CritPair} like for the S-pairs
from first two classes. But some S-pairs from the third class can
have a signature-greater generator polynomial $r_{l}$ such that $\totaldeg(r_{l})=\totaldeg(g),\,\Sig(r_{l})\prec\Sig(g)$.
They are processed with the second mentioned processing path so the
processing for such S-pairs is not yet started at the beginning of
last \texttt{Reduction} call. Fortunately, their processing starts
inside \texttt{Reduction} before fixed moment we studying: the procedure
selects polynomials from $ToDo$ in the signature increasing order,
so $r_{l}$ is reduced before $g$ and during $r_{l}$ reduction just
before putting $r_{l}$ to $Done$ a call to \texttt{IsReducible}
starts processing for all such S-pairs.
\end{proof}
The ideas of \emph{satisfying F5 criterion }and \emph{satisfying Rewritten
criterion} can be extended to arbitrary monomial-multiplied labeled
polynomial $sh,\, h\in G_{g}$:
\begin{defn}
The monomial-multiplied labeled polynomial $sr_{i},\, r_{i}\in G_{g}$
is called \emph{satisfying F5 criterion} if $\varphi_{index(r_{i})+1}(s\Sig(r_{i}))\ne s\Sig(r_{i})$,
where $\varphi_{index(r_{i})+1}$ is operator of normal form w.r.t
$G_{index(r_{i})+1}$.
\end{defn}
This definition is equivalent to $sr_{i}$ being non-normalized labeled
polynomial according to definition 2 in part 5 of \cite{FaugereF5}.
\begin{defn}
The monomial-multiplied labeled polynomial $sr_{i},\, r_{i}\in G_{g}$
is called \emph{satisfying Rewritten criterion} if $\exists j>i$
such that $\Sig(r_{j})|s\Sig(r_{i})$.
\end{defn}
For the case $sr_{i}$ is the S-pair part these definitions are equivalent
to the checks in the algorithm in a sense that S-pair part is rejected
by the algorithm if and only if it satisfies the definition as monomial-multiplied
labeled polynomial. Note that for both criteria holds important property
that if $sr_{i}$ satisfies a criteria then a further multiplied $s_{1}sr_{i}$
$ $satisfies it too.

\section{Representations}

\subsection{Definition}

The idea of representations comes from \cite{F5C}, where a similar
method is used in the proof of Theorem 21. Representations are used
to describe all possible ways how a labeled polynomial $p$ can be
written as an element of $\left(G_{g}\right)$ ideal. The single representation
corresponds to writing a labeled polynomial $p$ as any finite sum
of the form
\begin{equation}
p=\sum_{k}m_{k}\cdot b_{i_{k}},\; b_{i_{k}}\in G_{g}\label{eq:Gg-repr-def}
\end{equation}
with coefficients $m_{k}=c_{k}t_{k}\in\mathcal{K}\times T$. 
\begin{defn}
Sum of the form \ref{eq:Gg-repr-def} with all pairs $\left(t_{k},b_{i_{k}}\right)$
distinct is called \emph{$G_{g}$-representation} of $p$. The symbolic
products $m_{k}\cdot b_{i_{k}}$ are called the \emph{elements} of
representation. If we treat this symbolic product as multiplication
we get an labeled polynomial $m_{k}b_{i_{k}}$ corresponding to the
representation element. So $p$ is equal to sum of labeled polynomials,
corresponding to elements of its representation. Also the term \emph{element
signature} will be used for signature of labeled polynomials corresponding
to the element. Two representations are equal if the sets of their
elements are equal.
\end{defn}
Most representations we are interested in have the following additional
property limiting elements signature:
\begin{defn}
The $G_{g}$-representation of $p$ is called \emph{signature-safe}
if $\forall k\,\Sig(m_{k}b_{i_{k}})\preccurlyeq\Sig(p)$.
\end{defn}

\subsection{Examples}
\begin{example}
The first important example of a $G_{g}$-representation is trivial:
the labeled polynomial from $G_{g}$ is equal to sum of one element,
identity-multiplied itself: 
\[
b_{j}=1\cdot b_{j}.
\]

\end{example}
This $G_{g}$-representation is signature-safe. The prohibition of
two elements which have same monomial $t_{k}$ and polynomial $b_{i_{k}}$
ensures that all elements of representation that differ only in field
coefficient $c_{k}$ are combined together by summing field coefficients.
So expressions like $b_{j}=-1\cdot b_{j}+2\cdot b_{j}$ and $b_{j}=2x\cdot b_{k}+1\cdot b_{j}-2x\cdot b_{k}$
are not valid $G_{g}$-representations.
\begin{example}
A labeled polynomial $b_{j}\in G_{g}$ multiplied by arbitrary polynomial
$h$ also have a simple $G_{g}$-representation arising from splitting
$h$ into terms: $h=\sum_{k}m_{k},\, m_{k}\in\mathcal{K}\times T$.
This $G_{g}$-representation has form 
\begin{equation}
b_{j}h=\sum_{k}m_{k}\cdot b_{j}\label{eq:repr-ex-2}
\end{equation}
and is signature-safe too.
\end{example}
A labeled polynomial can have arbitrary number of representations:
for example we can add elements corresponding to a syzygy to any representation
and combine elements with identical monomials and polynomials to get
the correct representation. The result will be representation of the
same polynomial because sum of syzygy elements is equal to 0.
\begin{example}
The product of two polynomial from $G_{g}$ has two representations
of the form (\ref{eq:repr-ex-2}) which differs in syzygy addition:
\end{example}
\[
b_{j}b_{i}=\sum_{k}m_{i_{k}}\cdot b_{j}=\sum_{k}m_{i_{k}}\cdot b_{j}+0=\sum_{k}m_{i_{k}}\cdot b_{j}+\left(\sum_{k}m_{j_{k}}\cdot b_{i}-\sum_{k}m_{i_{k}}\cdot b_{j}\right)=\sum_{k}m_{j_{k}}\cdot b_{i},
\]
where $m_{i_{k}}$ are terms of $b_{i}$ and $m_{j_{k}}$ are terms
of $b_{j}$.
\begin{example}
The zero polynomial has an empty representation and an representation
for every syzygy:
\end{example}
\[
0=\sum_{\emptyset}\mbox{(empty sum)}=\sum_{k}m_{j_{k}}\cdot b_{i}+\sum_{k}(-m_{i_{k}})\cdot b_{j},
\]
where $m_{i_{k}}$ and $m_{j_{k}}$ are same as above.

Another important example of $G_{g}$-representation comes from ideal
and signature definitions. All labeled polynomials computed by the
algorithm are elements of ideal $\left(f_{1},\ldots,f_{m}\right)$.
So any labeled polynomial $p$ can be written as $ $$\sum_{i}f_{i}g_{i}$,
where $g_{i}$ are homogeneous polynomials. All input polynomials
$f_{i}$ belong to $G_{g}$, so $f_{i}g_{i}$ has $G_{g}$-representations
of the form (\ref{eq:repr-ex-2}).
\begin{example}
Those representations sum give the following signature-safe representation:
\[
p=\sum_{k}m_{k}\cdot b_{i_{k}},\, m_{k}\in\mathcal{K}\times T,\, b_{i_{k}}\in\left\{ f_{1},\ldots,f_{m}\right\} \subset G_{g}.
\]
\end{example}
\begin{defn}
This particular case of $G_{g}$-representation where $b_{i_{k}}$
are limited to input polynomials will be called \emph{input-representation}.
\end{defn}
Input representations always has the only element with maximal signature.
This property is special to input-representations because generic
$G_{g}$-representations can have multiple elements with same maximal
signature -- it is possible to have $m_{1}\Sig(b_{i_{1}})=m_{2}\Sig(b_{i_{2}})$
while $i_{1}\ne i_{2}$.

The following claim makes important connection between signatures
and input-representations:
\begin{claim}
An admissible labeled polynomial $p$ with known signature $\Sig(p)$
has an input-representation consisting of an element $c\Sig(p)\cdot f_{index(p)}$
and some other elements with smaller signatures.\end{claim}
\begin{proof}
The claimed fact follows from the admissible polynomial definition
in \cite{FaugereF5} referring to function $v$ which correspond to
summing representation elements.
\end{proof}
The theorem 1 of \cite{FaugereF5} states that all polynomials in
the algorithm are admissible, do the above claim will be applied to
all appeared polynomials.
\begin{example}
\label{example-of-having-gg-repr}The last example comes from S-pairs
with a computed $G_{g}$-representation. S-polynomial of $b_{l^{*}}$
and $b_{l_{*}}$ from $G_{g}$ is $p=\overline{u_{l}}\poly(b_{l^{*}})-\underline{u_{l}}\poly(b_{l_{*}})$.
It is known from reduction process that for such S-pairs $p$ is signature-safe
reduced and the result is added to $G_{g}$ as some labeled polynomial
$b_{l}$. So the $G_{g}$-representation is:
\end{example}
\[
p=\sum_{k}m_{k}\cdot b_{n_{k}}+1\cdot b_{l},
\]
where signatures of $m_{k}\cdot b_{n_{k}}$ elements are smaller than
$\Sig(b_{l})=\Sig(p)$. The value of $l$ is position of $b_{l}$
in ordered list $G_{g}$. In this representation $l$ is greater than
$l^{*}$ and $l_{*}$ because corresponding labeled polynomial $b_{l}$
is added to $R$ at the moment of S-polynomial computation in \texttt{Spol}
or \texttt{TopReduction} so the polynomials $b_{l^{*}}$ and $b_{l_{*}}$
used to create S-pair already present in $R$ at that moment and the
order of $G_{g}$ correspond to order of $R$.

\subsection{Ordering representations}
\begin{defn}
To order $G_{g}$-representations we start from \emph{representation
elements ordering} $\gtrdot_{1}$: we say that $c_{i}t_{i}\cdot b_{i}\gtrdot_{1}c_{j}t_{j}\cdot b_{j}$
if one of the following cases holds:
\begin{itemize}
\item $t_{i}\Sig(b_{i})\succ t_{j}\Sig(b_{j})$ 
\item $t_{i}\Sig(b_{i})=t_{j}\Sig(b_{j})$ and $i<j$ (note the opposite
order).
\end{itemize}
\end{defn}
This ordering is based only on comparison of signatures and positions
of labeled polynomials in the ordered list $G_{g}$ but does not depend
on the field coefficient. The only case in which two elements can't
be ordered is equality of both signatures $t_{i}\Sig(b_{i})=t_{j}\Sig(b_{j})$
and positions in list $i=j$. Position equality means $b_{i}=b_{j}$
which in conjunction with signature equality gives $t_{i}=t_{j}$.
So any two elements that belong to a single $G_{g}$-representation
are comparable with $\lessdot_{1}$ order because they have distinct
$\left(t_{k},b_{k}\right)$ by definition. Below are given some examples
of $\lessdot_{1}$ element ordering for the 3-element list $G_{g}=\left\{ b_{1},\, b_{2},\, b_{3}\right\} $
with ordering $x\mathbf{F}_{i}\succ y\mathbf{F}_{i}$ and signatures
$\Sig(b_{1})=\mathbf{F}_{1},\Sig(b_{2})=\mathbf{F}_{2},\Sig(b_{3})=x\mathbf{F}_{1}$.
\begin{itemize}
\item $y\cdot b_{1}\gtrdot_{1}100y\cdot b_{2}$ because signature of left
side is $\succ$
\item $x\cdot b_{1}\gtrdot_{1}y\cdot b_{1}$ because signature of left side
is $\succ$
\item $-x\cdot b_{1}$ and $2x\cdot b_{1}$ are not comparable because signatures
and list indexes are equal
\item $y^{2}\cdot b_{1}\lessdot_{1}y\cdot b_{3}$ because signature of left
side is $\prec$
\item $x^{2}\cdot b_{1}\gtrdot_{1}x\cdot b_{3}$ because signatures are
equal and the list position of left side's labeled polynomial is
1 which is smaller than right side's position 3.
\end{itemize}
To extend this order to entire $G_{g}$-representations consider \emph{ordered
form} of representation consisting of all its elements written in
a list with $\gtrdot_{1}$-decreasing order. This form can be used
for equality testing because if two representations are equal then
they have exactly equal ordered forms.
\begin{defn}
With ordered forms the \emph{$G_{g}$-representations ordering} can
be introduced: the representation $\sum_{k}m'_{k}\cdot b_{i'_{k}}$
is $\lessdot$-smaller than $\sum_{k}m_{k}\cdot b_{i_{k}}$ if the
ordered form of the first representation is smaller than second's
according to lexicographical extension of $\lessdot_{1}$ ordering
on elements. For the corner case of the one ordered form being beginning
of the other the shorter form is $\lessdot$-smaller. If the greatest
different elements of ordered forms differ only in field coefficient
the representations are not comparable.
\end{defn}
Some examples of this ordering are given for the same as above 3-element
$G_{g}$ list. Note that all $G_{g}$-representations are already
written in ordered forms:
\begin{itemize}
\item $x^{2}\cdot b_{1}+xy\cdot b_{1}+y^{2}b_{1}\gtrdot x^{2}\cdot b_{1}+100y^{2}\cdot b_{1}$
because $xy\cdot b_{1}\gtrdot y^{2}\cdot b_{1}$
\item $x^{2}\cdot b_{1}+100y^{2}\cdot b_{1}\gtrdot x^{2}\cdot b_{1}$ because
the right ordered form is beginning of the left
\item $x^{2}\cdot b_{1}\gtrdot xy\cdot b_{1}+y^{2}\cdot b_{1}+x^{2}\cdot b_{2}$
because $x^{2}\cdot b_{1}\gtrdot xy\cdot b_{1}$
\item $xy\cdot b_{1}+y^{2}\cdot b_{1}+x^{2}\cdot b_{2}\gtrdot y\cdot b_{3}+y^{2}\cdot b_{1}+x^{2}\cdot b_{2}$
because $xy\cdot b_{1}\gtrdot y\cdot b_{3}$
\item $y\cdot b_{3}+y^{2}\cdot b_{1}+x^{2}\cdot b_{2}$ and $2y\cdot b_{3}+y^{2}\cdot b_{2}$
are not comparable because the greatest different elements are $y\cdot b_{3}$
and $2y\cdot b_{3}$.
\end{itemize}
The ordering is compatible with signature-safety:
\begin{thm}
If two representations of $p$ has a relation $\sum_{k}m'_{k}\cdot b_{i'_{k}}\lessdot\sum_{k}m_{k}\cdot b_{i_{k}}$
and the second one is signature-safe representation then the first
one is signature-safe too.\end{thm}
\begin{proof}
This theorem quickly follows from a fact that elements of a $\lessdot$-smaller
representation can't has signatures $\succ$-greater than signatures
of $\gtrdot$-greater representation.
\end{proof}
The key fact allowing to take $\lessdot$-minimal element is well-orderness:
\begin{thm}
The representations are well-ordered with $\lessdot$ ordering.\end{thm}
\begin{proof}
The number of different labeled polynomial positions is finite because
it is equal to $|G_{g}|$ which is finite for fixed $g$. So the existence
of infinite $\gtrdot_{1}$-descending sequence of representation elements
would lead to existence of infinite $\succ$-descending sequence of
signatures. Combining this with well-orderness of signatures with
ordering $\prec$ we get the proof for well-orderness of elements
with ordering $\lessdot{}_{1}$.

The straightforward proof for $\lessdot$-well-orderness of representations
following from $\lessdot_{1}$-well-orderness of elements is not very
complex but to skip its strict details the theorem 2.5.5 of \cite{Baader.Nipkow98Term}
will be referenced. It states well-orderness of finite multiset with
an lexicographically extended ordering of well-ordered elements. This
applies to the representations because they form a subset in the finite
multiset of representation elements.
\end{proof}

\subsection{Sequence of representations}

The idea of this part is constructing a finite sequence of strictly
$\lessdot$-descending signature-safe $G_{g}$-representations for
a given labeled polynomial $mh,\, m\in\mathcal{K}\times T,\, h\in G_{g}$
with $\Sig(mh)\prec\Sig(g)$. The first signature-safe representation
in the sequence is $mh=m\cdot h$, the last representation is $mh=\sum_{k}m_{k}\cdot b_{i_{k}}$
with elements having the following properties $\forall k$:
\begin{enumerate}
\item $m_{k}b_{i_{k}}$ does not satisfy F5 criterion.
\item $m_{k}b_{i_{k}}$ does not satisfy Rewritten criterion.
\item $\HM(m_{k}b_{i_{k}})\leqslant\HM(mh)$
\end{enumerate}
The proof of such sequence existence is very similar to Theorem 21
of \cite{F5C} and is based on a fact, that if a some signature-safe
representation of $mh$ contains an element $m_{K}\cdot b_{i_{K}}$
not having one of the properties then a $\lessdot$-smaller representation
can be constructed. The exact construction differ for three cases
but the replacement scheme is the same:
\begin{itemize}
\item a some element $m_{K'}\cdot b_{i_{K'}}$ in $mh$ representation is
selected. Note that $K'$ in some cases is not equal to $K$
\item some representation $m_{K'}b_{i_{K'}}=\sum_{l}m_{l}\cdot b_{i_{l}}$
is constructed for this element.
\item it is shown that constructed representation is $\lessdot$-smaller
than representation $m_{K'}b_{i_{K'}}=m_{K'}\cdot b_{i_{K'}}$ 
\end{itemize}
Construction of such representation for $m_{K'}\cdot b_{i_{K'}}$
allows application of the following lemma:
\begin{lem}
If an element $m_{K'}\cdot b_{i_{K'}}$ of signature-safe representation
$mh=\sum_{k}m_{k}\cdot b_{i_{k}}$ has an representation $m_{K'}b_{i_{K'}}=\sum_{l}m_{l}\cdot b_{i_{l}}$
which is $\lessdot$-smaller than representation $m_{K'}b_{i_{K'}}=m_{K'}\cdot b_{i_{K'}}$
then $mh$ has a signature-safe representation $\lessdot$-smaller
than $mh=\sum_{k}m_{k}\cdot b_{i_{k}}$.\end{lem}
\begin{proof}
We replace $m_{K'}\cdot b_{i_{K'}}$ in $mh=\sum_{k}m_{k}\cdot b_{i_{k}}$
by $\sum_{l}m_{l}\cdot b_{i_{l}}$ and combine coefficients near elements
with both monomial and polynomial equal, so a modified representation
for $mh$ appears. Is is $\lessdot$-smaller than $mh=\sum_{k}m_{k}\cdot b_{i_{k}}$
because all elements $\gtrdot_{1}$-greater than $m_{K'}\cdot b_{i_{K'}}$
are identical in both representations if they present but the element
 $m_{K'}\cdot b_{i_{K'}}$ is contained in original representation
but not in the modified. And all other elements in representations
are $\lessdot_{1}$-smaller than $m_{K'}\cdot b_{i_{K'}}$ so they
does not influence the comparison. The comparison holds even in a
corner case when all elements are discarded while combining coefficients.
This case can appear if the original representation is equal to $mh=m_{K'}\cdot b_{i_{K'}}+\sum_{l}(-m_{l})\cdot b_{i_{l}}$
what leads to modified representation $mh=0$ with zero elements which
is $\lessdot$-smaller than any non-empty representation.
\end{proof}
Now it will be shown that replacement scheme can be performed if the
representation contains an element not satisfying at least one of
three properties.
\begin{lem}
If a signature-safe $G_{g}$-representation $mh=\sum_{k}m_{k}\cdot b_{i_{k}}$
does not satisfy property 1 then there exists an element $m_{K'}\cdot b_{i_{K'}}$
having $G_{g}$-representation $m_{K'}b_{i_{K'}}=\sum_{l}m_{l}\cdot b_{i_{l}}$
which is $\lessdot$-smaller than representation $m_{K'}b_{i_{K'}}=m_{K'}\cdot b_{i_{K'}}$
.
\end{lem}
An element not having the first property does satisfy the F5 criterion
and the idea is to use that $m_{K}\Sig(b_{i_{K}})$ is not the minimal
signature of $m_{K}b_{i_{K}}$ like in Theorem 20 of \cite{F5C}.
$K'=K$ is taken for this case.
\begin{proof}
Consider input-representation of $m_{K}b_{i_{K}}$ with signature
of $\gtrdot_{1}$-maximal element equal to $m_{K}\Sig(b_{i_{K}})=s_{0}\mathbf{F}_{j_{0}}$:

\begin{equation}
m_{K}b_{i_{K}}=c_{0}s_{0}\cdot f_{j_{0}}+\sum_{l}m_{l}\cdot f_{i_{l}}.\label{eq:input-repr-case1}
\end{equation}
From the satisfying F5 criterion $s_{0}$ can be expressed like $s_{0}=s_{1}\HM(f_{j_{1}}),\, j_{1}>j_{0}$
so $s_{0}f_{j_{0}}=s_{1}f_{j_{0}}f_{j_{1}}-s_{1}(f_{j_{1}}-\HM(f_{j_{1}}))f_{j_{0}}$.
From this we can write another representation for $m_{K}b_{i_{K}}$,
assuming $m_{0i}$ are sorted terms of $f_{j_{0}}$, $m_{1i}$ are
sorted terms of $f_{j_{1}}$ and $N_{0,}N_{1}$ are number of terms
in those polynomials:

\[
m_{K}b_{i_{K}}=\sum_{i=1}^{N_{0}}c_{0}s_{1}m_{0i}\cdot f_{j_{1}}+\sum_{i=2}^{N_{1}}-c_{0}s_{1}m_{1i}\cdot f_{j_{0}}+\sum_{l}m_{l}\cdot f_{i_{l}}.
\]
This representation is $\lessdot$-smaller than $m_{K}\cdot b_{i_{K}}$
because signatures of all elements are smaller than $s_{0}\mathbf{F}_{j_{0}}$.
For the elements of the third sum $\sum_{l}m_{l}\cdot f_{i_{l}}$
this follows from \textbf{\ref{eq:input-repr-case1}}, where those
elements are smaller elements of input-representation. For the elements
of the first sum $\sum_{i=1}^{N_{0}}c_{0}s_{1}m_{0i}\cdot f_{j_{1}}$
this follows from the position inequality $j_{1}>j_{0}$. And for
the second sum we use the equality in term and signature orderings:
all terms $m_{1i},\, i\geqslant2$ are smaller than $m_{11}$, so
the signatures are: $s_{1}m_{1i}\mathbf{F}_{j_{0}}\prec s_{1}m_{11}\mathbf{F}_{j_{0}}=s_{0}\mathbf{F}_{j_{0}}$. \end{proof}
\begin{lem}
If a signature-safe $G_{g}$-representation $mh=\sum_{k}m_{k}\cdot b_{i_{k}}$
with $\Sig(mh)\prec\Sig(g)$ does not satisfy property 2 then there
exists an element $m_{K'}\cdot b_{i_{K'}}$ having $G_{g}$-representation
$m_{K'}b_{i_{K'}}=\sum_{l}m_{l}\cdot b_{i_{l}}$ which is $\lessdot$-smaller
than representation $m_{K'}b_{i_{K'}}=m_{K'}\cdot b_{i_{K'}}$ \textup{.}
\end{lem}
For the elements not satisfying case 2 the $\lessdot$-smaller representation
is created in a way used in Proposition 17 of \cite{F5C}. $K'=K$
is taken for this case too.
\begin{proof}
Assume that $\Sig(m_{K}b_{i_{K}})=s_{0}\mathbf{F}_{j_{0}}$ and it
is rewritten by labeled polynomial $b_{i'}$ from $R$. Because the
representation is signature-safe we have $\Sig(b_{i'})\preccurlyeq s_{0}\mathbf{F}_{j_{0}}\preccurlyeq\Sig(mh)\prec\Sig(g)$.
So $b_{i'}$ was processed in \texttt{TopReduction} before $g$. Since
$b_{i'}$ is rewriter we have $b_{i'}\ne0$. All this gives the fact
that $b_{i'}$ does present not only in $R$ but in $G_{g}$ too so
it can be used as a polynomial of $G_{g}$-representation element.
From the Rewritten criterion definition we know that $i'>i_{K}$ and
the existence of $s'\in T$ such that $s'\Sig(b_{i'})=s_{0}\mathbf{F}_{j_{0}}$.
So, for the $m_{K}b_{i_{K}}$ there is an input-representation \ref{eq:input-repr-case1}
and for the $s'b_{i'}$ the input-representation is:

\[
s'b_{i'}=c's_{0}\cdot f_{j_{0}}+\sum_{l'}m_{l'}\cdot f_{i_{l'}}.
\]
A $G_{g}$-representation for $c_{0}s_{0}f_{j_{0}}$ can be acquired
with transformation of the above expression:

\[
c_{0}s_{0}f_{j_{0}}=c'^{-1}c_{0}s'\cdot b_{i'}+\sum_{l'}-c'^{-1}c_{0}m_{l'}\cdot f_{i_{l'}}.
\]
Using this to replace the first element in \ref{eq:input-repr-case1}
we get the wanted result:
\[
m_{K}b_{i_{K}}=c'^{-1}c_{0}s'\cdot b_{i'}+\sum_{l'}-c'^{-1}c_{0}m_{l'}\cdot f_{i_{l'}}+\sum_{l}m_{l}\cdot f_{i_{l}}
\]
It is $\lessdot$-smaller than $m_{K}b_{i_{K}}=m_{K}\cdot b_{i_{K}}$
because elements of both sums has signatures smaller than $s_{0}\mathbf{F}_{j_{0}}$,
and for the first element $ $$\Sig(c'^{-1}c_{0}s'\cdot b_{i'})=\Sig(m_{K}\cdot b_{i_{K}})=s_{0}\mathbf{F}_{j_{0}}$
but $i'>i_{K}$, so applying the $\lessdot_{1}$-comparison rule for
equal signatures and different list positions we get that element
$c'^{-1}c_{0}s'\cdot b_{i'}$ is $\lessdot_{1}$-smaller than $m_{K}\cdot b_{i_{K}}$
too.\end{proof}
\begin{lem}
If a signature-safe representation $mh=\sum_{k}m_{k}\cdot b_{i_{k}}$
with $\Sig(mh)\prec\Sig(g)$ satisfies properties 1 and 2 but does
not satisfy property 3 then there exists an element $m_{K'}\cdot b_{i_{K'}}$
having representation $m_{K'}b_{i_{K'}}=\sum_{l}m_{l}\cdot b_{i_{l}}$
which is $\lessdot$-smaller than representation $m_{K'}b_{i_{K'}}=m_{K'}\cdot b_{i_{K'}}$.\end{lem}
\begin{proof}
There exists at least one element $m_{K}\cdot b_{i_{K}}$ that does
not satisfy property 3. Let $m_{\max}$ be the maximal $\HM$ of labeled
polynomials corresponding to representation elements and $H_{\max}$
be a list of elements where $m_{\max}$ is achieved. Select $K'$
to be the index of the $\gtrdot_{1}$-greatest representation element
in $H_{\max}$. We have $ $$\HM(m_{K'}b_{i_{K'}})=m_{\max}\geqslant\HM(m_{K}b_{i_{K}})>\HM(mh)$,
so the $\HM$ of sum of all elements except $K'$ is equal to $\HM(mh-m_{K'}b_{i_{K'}})=\HM(m_{K'}b_{i_{K'}})=m_{\max}$,
so there is another element $K''$ having $\HM(m_{K''}b_{i_{K''}})=m_{\max}$.
So, $m_{K''}\cdot b_{i_{K''}}\in H_{\max}$ and $m_{K''}\cdot b_{i_{K''}}\lessdot_{1}m_{K'}\cdot b_{i_{K'}}$
because of $ $$m_{K'}\cdot b_{i_{K'}}$ $\gtrdot_{1}$-maximality
in $H_{\max}$.

The $\HM(m_{K''}b_{i_{K''}})=\HM(m_{K'}b_{i_{K'}})$ means that a
critical pair of $b_{i_{K'}}$ and $b_{i_{K''}}$ has the form $[m'^{-1}m_{\max},\, m'^{-1}m_{K'},\, b_{i_{K'}},\, m'^{-1}m_{K''},\, b_{i_{K''}}]$
where $m'=\mbox{gcd}(m_{K'},m_{K''})$. Let $q$ be corresponding
S-polynomial. Then $m'\Sig(q)\preccurlyeq\Sig(mh)\prec\Sig(g)$ because
$m'\Sig(q)=\Sig(m_{K'}b_{i_{K'}})$ and the representation is signature-safe.
The S-polynomial parts $m'^{-1}m_{K'}b_{i_{K'}}$ and $m'^{-1}m_{K''}b_{i_{K''}}$
does not satisfy F5 and Rewritten criteria because their forms multiplied
by $m'$ are $m_{K'}b_{i_{K'}}$ and $m_{K''}b_{i_{K''}}$ -- labeled
polynomials corresponding to elements which are known not to satisfy
both criteria by assumption. Therefore $m'\Sig(q)\prec\Sig(g)$ and
$\Sig(q)\prec\Sig(g)$. It follows from this with theorem \ref{thm:Exist-gg-repr}
that the S-pair $(b_{i_{K'}},b_{i_{K''}})$ is S-pair with computed
$G_{g}$-representation, what means that there is an representation
described in example \ref{example-of-having-gg-repr} :

\[
q=1\cdot b_{i'}+\sum_{l}m_{l}\cdot b_{i_{l}},
\]
satisfying the properties shown after that example: $\Sig(q)=\Sig(b_{i'})$,
$\forall l\,\Sig(q)\succ\Sig(m_{l}b_{i_{l}})$ and $i'>K'$.

From the other hand we have $m'q=c_{0}m_{K'}b_{i_{K'}}-c_{1}m_{K''}b_{i_{K''}}$,
so we get the following representation:
\[
m_{K'}b_{i_{K'}}=c_{0}^{-1}c_{1}m_{K''}\cdot b_{i_{K''}}+c_{0}^{-1}m'\cdot b_{i'}+\sum_{l}c_{0}^{-1}m'm_{l}\cdot b_{i_{l}}.
\]
It is $\lessdot$-smaller than $m_{K'}b_{i_{K'}}=m_{K'}\cdot b_{i_{K'}}$: 

$m_{K''}\cdot b_{i_{K''}}$ was already compared to $m_{K'}\cdot b_{i_{K'}}$ 

$m'\cdot b_{i'}$ has the same signature but greater position $i'>i_{K'}$

the last sum contains elements with signatures smaller than $m'\Sig(b_{i'})=\Sig(m_{K'}\cdot b_{i_{K'}})$.\end{proof}
\begin{thm}
\label{thm:exist-smaller-signature-safe-representation}A signature-safe
representation $mh=\sum_{k}m_{k}\cdot b_{i_{k}}$ with $\Sig(mh)\prec\Sig(g)$
either satisfies properties 1-3 or there exists a signature-safe representation
$mh=\sum_{l}m_{l}\cdot b_{i_{l}}$ which is $\lessdot$-smaller than
\textup{$\sum_{k}m_{k}\cdot b_{i_{k}}$.}\end{thm}
\begin{proof}
This theorem quickly follows from four previous lemmas together
\end{proof}
This leads to main result:
\begin{thm}
For any labeled polynomial $mh,\, m\in\mathcal{K}\times T,\, h\in G_{g}$
with $\Sig(mh)\prec\Sig(g)$ there exists a signature-safe $G_{g}$-representation
$mh=\sum_{k}m_{k}\cdot b_{i_{k}}$ that satisfies properties 1-3\textup{.}\end{thm}
\begin{proof}
Start with representation $mh=m\cdot h$ and begin replacing it by
$\lessdot$-smaller representation from theorem \ref{thm:exist-smaller-signature-safe-representation}
until the representation satisfying properties 1-3 appears. The finiteness
of the process is guaranteed by $\lessdot$-well-orderness.
\end{proof}
This result may be interesting by itself, but for the purposes of
proving termination only one corollary is needed:
\begin{cor}
\label{cor:all-needed-for-terminaton}Consider an arbitrary polynomial
$f$ without any restrictions on its signature. If there exists a
signature-safe reductor $f'\in G_{g}$ for $f$ with $\Sig(f')\frac{\HM(f)}{\HM(f')}\prec\Sig(g)$
then $G_{g}$ contains a signature-safe reductor for $f$ that is
not rejected by F5 and Rewritten criteria.\end{cor}
\begin{proof}
Let $mf',\, m=\frac{\HM(f)}{\HM(f')}\in\mathcal{K}\times T,\, f'\in G_{g}$
be a multiplied reductor with $\Sig(mf')\prec\Sig(g)$. From the previous
theorem we can find representation $mf'=\sum_{k}m_{k}\cdot b_{i_{k}}$
that satisfies properties 1-3. Property 3 means that there is no elements
with $\HM$'s greater than $mf'$ so because sum of all elements
has $\HM$ equal to $\HM(mf')$ there exists an element $K$ that
achieves $\HM$ equality: $\HM(m_{K}\cdot b_{i_{K}})=\HM(mf')=\HM(f_{1}')$.
Since the representation is signature-safe $\Sig(m_{K}\cdot b_{i_{K}})\preccurlyeq\Sig(mf')\prec\Sig(f)$
so $m_{K}b_{i_{K}}$ is a signature-safe reductor for $f$ and properties
1-2 ensure that $m_{K}b_{i_{K}}$ does not satisfy criteria.
\end{proof}

\section{Finding contradiction with the criteria enabled}

Now return to the result of theorem \ref{thm:f_g_3_props} which states
for the case of algorithm non-termination existence of a polynomials
$f',f\in G$ such that $\HM(f')|\HM(f)$, $\frac{\HM(f')}{\Sig(f')}>_{q}\frac{\HM(f)}{\Sig(f)}$.
Using this result and last corollary we construct two polynomials
leading to contradiction for the case of algorithm non-termination.
\begin{thm}
\label{thm:always-exist-ok-reductor}If the algorithm does not terminate
for some input then after some finite step the set \textup{$G\cup Done$}
contains a pair of labeled polynomials $f'_{1},f$ where:
\begin{itemize}
\item $f'_{1}$ is added to \textup{$G\cup Done$} before $f$
\item $t_{1}f'_{1}$ does not satisfy F5 and Rewritten criteria, where $t_{1}=\frac{\HM(f)}{\HM(f'_{1})}$
\item $f'_{1}$ is signature-safe reductor for $f$.
\end{itemize}
\end{thm}
\begin{proof}
Let $f',f$ be polynomials from the theorem \ref{thm:f_g_3_props}
an define $t=\frac{\HM(f)}{\HM(f')}$. We have $f\in G$ so the above
theory about representations can be applied to the fixed value of
$g$ equal to $f$ and we can speak about $G_{f}$ set and $G_{f}$-representations.
Because $tf'$ is a signature-safe reductor for $f$ we have $\Sig(f')t\prec\Sig(f)$
and the corollary \ref{cor:all-needed-for-terminaton} can be applied
to find a signature-safe reductor $t_{1}f'_{1}$ for $f$ which does
not satisfy criteria. Also it is known to belong to $G_{f}$, so during
the algorithm execution $f'_{1}$ was appended to $G\cup Done$ before
$f$.\end{proof}
\begin{thm}
The original $ $F5 algorithm as described in \cite{FaugereF5} does
terminate for any input.\end{thm}
\begin{proof}
We are going o show that the existence of polynomials $f'_{1},f$
from the theorem \ref{thm:always-exist-ok-reductor} leads to contradiction.
Consider the call to \texttt{TopReduction} after which the polynomial
$f$ was inserted in $Done$. That call returns polynomial $f$ as
first part of \texttt{TopReduction} return value, so the value returned
by \texttt{IsReducible} is empty set. It means that one of conditions
(a) - (d) was not satisfied for all polynomials in $G\cup Done$ including
$f'_{1}$. This is not possible because:
\begin{itemize}
\item (a) is satisfied because $f'_{1}$ is a reductor for $f$ from the
theorem \ref{thm:always-exist-ok-reductor} 
\item (b) and (c) are satisfied because $\frac{\HM(f)}{\HM(f'_{1})}f'_{1}$
does not satisfy F5 and Rewritten criteria from the theorem \ref{thm:always-exist-ok-reductor} 
\item (d) is satisfied because $f'_{1}$ is a signature-safe reductor for
$f$ from the theorem \ref{thm:always-exist-ok-reductor}.
\end{itemize}
\end{proof}

\section{Conclusions}

This paper shows that original F5 algorithm terminates for any homogeneous
input without introducing intermediate algorithms. However, it does
not give any limit on number of operations. The simplest proof of
the termination of Buchberger algorithm is based on Noetherian property
and does not give any such limit too. Unfortunately the termination
proof given here is quite different in structure compared to the proof
of Buchberger algorithm termination, so this proof does not show that
F5 is more efficient than Buchberger in any sense. Unlike this the
termination of the modified versions of F5 algorithm in \cite{Modifying-for-termination,Ars05applicationsdes,Gash:2008:ECG}
is shown in a way analogous to Buchberger algorithm and there is room
for comparison of their efficiency with Buchberger's one.

From the point of view of practical computer algebra computations
there is a question about efficiency of the modified versions compared
to original F5. The modified versions can spend more time in additional
termination checks. But for some cases it is possible that those checks
can allow the termination of modified versions before original so
the modified version performs smaller number of reductions. So it
is possible that for some inputs the original algorithm is faster
and for others the modified version. Some experimental timings in
Table 1 in \cite{Modifying-for-termination} shows that both cases
are possible in practice but the difference in time is insignificant.
So the question about efficiency of original F5 compared to modified
versions is open.

This proof uses three properties of original F5 that are absent or
optional in some F5-like algorithms: the homogeneity of input polynomials,
the presence of Rewritten criterion and the equality of monomial order
$<$ and signature order $\prec$. The possibility of extending the
termination proof to the modified algorithms without these properties
is open question. There is an unproved idea that the proof can be
modified to remove reliance on the first two properties but not on
the third property of orders equality because it is key point of coming
to a contradiction form the result of theorem \ref{thm:f_g_3_props}.\\

\thanks{The author would like to thank Christian Eder, Jean-Charles Faugère,
Amir Hashemi, John Perry, Till Stagers and Alexey Zobnin for inspiring
me on investigations in this area by their papers and comments. Thanks!}

\bibliographystyle{short}
\bibliography{f5_references}

\end{document}